\documentclass{amsart}

\usepackage{amsmath, amsthm, amssymb,}
\usepackage{color}

\usepackage[all]{xypic}  \entrymodifiers={+!!<0pt,\fontdimen22\textfont2>}

\swapnumbers \theoremstyle{plain}
\newtheorem{thm}{Theorem}[section]
\newtheorem{lemma}[thm]{Lemma}

\newtheorem{prop}[thm]{Proposition}

\theoremstyle{definition}

\newtheorem{EX}[thm]{Example}
\newtheorem{RM}[thm]{Remark}
\newtheorem{defn}[thm]{Definition}

\DeclareMathOperator{\CH}{\rm CH}

\newcommand{\bbQ}{{{\mathbb Q}}}  
\newcommand{\Z}{{{\mathbb{Z}}}}   

\newcommand{\calK}{\mathcal K}

\newcommand{\om}{\omega}

\title{Erratum: On the torsion of Chow groups of twisted Spin-flags}

\author{Sanghoon Baek}
\author{Kirill Zainoulline}
\author{Changlong Zhong}

\begin{document}

\maketitle


\section{List of mistakes and missprints}

\begin{enumerate}
\item
In \cite[Example 4.4]{BZZ}, for type $B_n$, $\{q_{2i}\}_{i=1}^n$ is a
set of basic polynomial invariants over $\bbQ$ only (not over $\mathbb{Z}[\tfrac{1}{2}]$). Similarly,
for type $D_n$, $\{q_{2i}\}_{i=1}^{n-1}\cup \{p_n\}$ is a set of basic
polynomial invariants over $\bbQ$ only.

\item Let $K$ denote the ideal generated by $\{q_{2i}\}_{i=1}^n$ (or
  $\{q_{2i}\}_{i=1}^{n-1}\cup \{p_n\}$) over $\mathbb{Z}$, then $K\subsetneq
  I_a^W$.  In \cite[Proposition 4.5]{BZZ} $I_a^W$ has to be replaced by $K$. 

\item Let $\tau_d'$ denote the smallest positive integer $N$ such that $N\cdot K^{(d)}\subset \phi^{(d)}(I_m^W)$.
In \cite[Proposition 5.8]{BZZ} $\tau_d$ has to be replaced by
$\tau_d'$. 

\item There is a gap in the proof of \cite[Lemma 4.6]{BZZ},  hence, so is in Proposition 4.7 in loc.it..
\end{enumerate}

\section{Corrections}
To correct our arguments instead of $\{q_{2i}\}$ we use another set of
generators $\{t_i\}$ (see below), which at the
end simplifies our proofs and provides even a better upper bound for the torsion of the Chow groups,
hence, improving our results.

\medskip

Let $t_{i}=s_i(e_1^2,\ldots,e_n^2)$ denote the $i$-th elementary symmetric
polynomial in $e_j^2$.  Then \cite[Example
4.4]{BZZ} has to be replaced by
\begin{EX}\label{ex:Fainv} \cite[Lemma 8.3]{MZZ} For the type $B_n$ ($n\ge 3$) we have 
$$S^*(\Lambda)^W\otimes \Z[\tfrac{1}{2}]=\Z[\tfrac{1}{2}][t_1,...,t_n],$$
and for the type $D_n$ ($n\ge 4$) we have 
$$S^*(\Lambda)^W\otimes \Z[\tfrac{1}{2}]=\Z[\tfrac{1}{2}][t_1,...,t_{n-1},p_n], \text{
  where }p_n=e_1\cdots e_n.$$

\end{EX}

Let $\calK$ be the ideal
generated by $t_1,\ldots,t_n$ over $\mathbb{Z}$
(resp. $t_1,\ldots,t_{n-1}, p_n$) if $\mathfrak{D}$ is of type $B_n$
(resp. $D_n$). 
Then \cite[Lemma 3.1, Corollaries 3.3 and 3.4]{BZZ} has to be replaced by

\begin{lemma}\label{QNcase} Suppose we have a root system of type $B_n$ (resp. $D_n$) and $d\ge 2$ (resp. $n>d\ge 2$). Assume that $d_0=\min\{2n,d\}$.
Consider a homogeneous polynomial $P=\sum_{i=1}^{[d_0/2]} f_{d-2i}t_{i}$
of degree $d$ with integer coefficients.

If $M\mid P$, then there exist $\hat f_{d-2i}$, $i=1,\ldots ,[d_0/2]$ such that
\[
\sum_{i=1}^{[d_0/2]} \hat{f}_{d-2i}t_{i}=P
\;\text{ and }\;M \mid \hat{f}_{d-2i}.
\]
\end{lemma}

\begin{proof}
For each $i$ we express $f_{d-2i}$ as a linear combination
\[
f_{d-2i}=\sum_\delta e^\delta f^\delta_{d-2i},
\]
where $\delta=(\delta_1,\ldots,\delta_n)$ with $\delta_i=0,1$, $ e^\delta=\prod_{i=1}^ne_i^{\delta_i}$
and $f^\delta_{d-2i}$ is a linear combination of even monomials $e_1^{2i_1}e_2^{2i_2}\ldots e_n^{2i_n}$ only. Denote $|\delta|=\sum \delta_j$, $d^\delta=\frac{d-|\delta|}{2}$ and $d_0^\delta=\min\{d^\delta, n\}\le [d_0/2]$.
Collecting the terms with $e^\delta$ we obtain
\[
P=\sum_\delta e^\delta \sum_{i=1}^{d^\delta_0} f^\delta_{d-2i}t_{i} \equiv 0 \mod M.
\]
This implies that  $M\mid \sum_{i=1}^{d^\delta_0} f^\delta_{d-2i}t_{i}$ for each $\delta$.
We apply \cite[Lemma 3.1]{BZZ} to the polynomial
$P_\delta=\sum_{i=1}^{d^\delta_0} f^\delta_{d-2i}t_{i}$
viewed as a polynomial in variables $e_{j}^2$
of degree $d^\delta$.
We obtain polynomials
$\hat  f^\delta_{d-2i}$ divisible by $M$ and
$\sum_{i=1}^{d^\delta_0} f^\delta_{d-2i}t_{i}=\sum_{i=1}^{d^\delta_0} \hat f^\delta_{d-2i}t_{i}$.

We then set
\[
\hat f_{d-2i}=\sum_\delta e^\delta \hat f^\delta_{d-2i},
\]
and the proof is finished.
\end{proof}

\cite[Prop. 4.5 and 4.7]{BZZ} have to be replaced by

\begin{prop}\label{Bcase} Let $\mathfrak{D}=\mathrm{B}_n$ (resp. $D_n$) and $d\ge 2$ (resp. $n>d\ge 2$). Then $2^{d}\cdot (\ker c_a)^{(d)}\subset \calK \subset I_a^W$.
\end{prop}
 
\begin{proof}
Let $d_0=\min\{d,2n\}$. By Example \ref{ex:Fainv}, $(\ker c_a) \otimes \Z[\tfrac{1}{2}]=I_a^W \otimes \Z[\tfrac{1}{2}]$ is generated by $t_{i}$, $i=1,...,n$,
given a polynomial $f\in (\ker c_a)^{(d)}$, we can write it as
\[
2^r f=\sum_{i=1}^{[d_0/2]} f_{d-2i}t_{i} \in I_a^W,\; \text{ for some }f_{d-2i}\in \Z[\om_1,...,\om_n]\text{ and }r \ge 0.
\]
Suppose $r$ is the smallest such integer.
To finish the proof it suffices to show that $r\le d$.

Assume the contrary, i.e. that $r\ge d+1$.
Expressing $\om_j$'s in terms of $e_j$'s, we obtain
$f=\tfrac{1}{2^{d}}\tilde{f}$ and $f_{d-2i}=\tfrac{1}{2^{d-2i}}\tilde{f}_{d-2i}$ for some $\tilde{f}, \tilde{f}_{d-2i}\in \Z[e_1,...,e_n]$.
So that
\[
 2^r \cdot \tilde f= \sum_{i=1}^{[d_0/2]} (2^{2i}\tilde{f}_{d-2i})\cdot t_{i}.
\]
By Corollary~\ref{QNcase}, there exists $\tilde{h}_{d-2i}\in \Z[e_1,...,e_n]$ divisible by $2^{d+1}$ such that
$2^r\cdot \tilde{f}=\sum_{i=1}^{[d_0/2]} \tilde{h}_{d-2i}t_{i}$.
Expressing $e_j$'s in terms of $\om_j$'s back, we obtain
$2^{d}2^r\cdot f=\sum_{i=1}^{[d_0/2]}  \tilde h_{d-2i}t_{i}$,
which implies
\[
2^{r-1}f=\sum_{i=1}^{[d_0/2]} \tfrac{1}{2^{d+1}}\tilde {h}_{d-2i}\cdot t_{i}.
\]
Since $\tilde{h}_{d-2i}$ are divisible by $2^{d+1}$, we have $\tfrac{1}{2^{d+1}}\tilde h_{d-2i} \in \Z[\om_1,...,\om_n]$.
This contradicts to the minimality assumption on $r$.
\end{proof}

\begin{RM}There is a typo in the proof of \cite[Lemma 6.4]{BNZ}, i.e., the congruence relations in the proof should be 
$$a_i\equiv a_j\equiv -d, ~a_{ij}-a_i\equiv a_{ij}-a_j\equiv 2d, ~a_i-a_{ij}+a_j\equiv -3d.$$
One can use Proposition~\ref{Bcase} above to get an upper bound of the index in degree 4. 
\end{RM}

If we modify the definition of the $d$-th exponent $\tau_d$ in
\cite[Definition 2.1]{BNZ} as follows

\begin{defn}
For type $B_n$ and type $D_n$, the $d$-th \textit{exponent} of a root system (denoted by $\tau_{d}$) is the g.c.d. of all nonnegative integers $N_{d}$ satisfying
\[N_d\cdot \calK^{(d)} \subseteq \phi^{(d)}(I_{m}^W),
\]
where $\calK^{(d)}=(\calK\cap I_a^d)/(\calK\cap I_a^{d+1})$.
\end{defn}
we see that the proof of \cite[Proposition 5.8]{BZZ} indeed implies that $\tau_d|2$ for all $d\le 2d-1$ (resp. $d\le 2d-3$) if $\mathfrak{D}$ is of type $B_n$ (resp. $D_n$). 
Then
following the proof of \cite[Theorem 6.2]{BZZ} we obtain that 
\begin{thm} The integer
$(i-1)!\cdot 2^{i+1}$ annihilates the torsion of $\gamma^{(i)}(G/B)$, and
$(d-1)!\prod_{i=2}^d(i-1)!\cdot 2^{i+1}$ annihilates the torsion of $\CH^d(_\xi(G/B))$.
\end{thm}


\begin{thebibliography}{99}

\bibitem{BNZ} S. Baek, E. Neher, K. Zainoulline, Basic polynomial invariants, fundamental representations and the Chern class map. Documenta Math. 17 (2012), 135-150.

\bibitem{BZZ} S. Baek, K. Zainoulline, C. Zhong, On the torsion of Chow groups of twisted spin-flags. Math. Res. Lett., to appear.

\bibitem{MZZ} J. Malag\'{o}n-L\'{o}pez, K. Zainoulline, C. Zhong, Invariants, exponents and formal group laws, arXiv: 1207.1880. 
\end{thebibliography}
\end{document}